\documentclass[12pt]{article}
\usepackage{graphicx} 
\usepackage[spanish,english]{babel}
\usepackage[top=1in, left=0.7 in, right=0.6 in,bottom=1in]{geometry}

\usepackage{natbib}
\usepackage{amsfonts,amssymb,amsmath,amsthm}
\usepackage{mathtools}
\mathtoolsset{showonlyrefs}
\newtheorem{theorem}{Theorem}
\theoremstyle{definition}
\newtheorem{proposition}{Proposition}
\theoremstyle{definition}
\newtheorem{lemma}{Lemma}
\newtheorem{definition}{Definition}
\newenvironment{lemmaproof}
  {\begin{proof}}
  {\end{proof}}

  \providecommand{\keywords}[1]
{
  \small	
  \textbf{\textit{Keywords---}} #1
}
\providecommand{\subjclass}[1]
{
  \small	
  \textbf{\textit{MSC numbers---}} #1
}

\title{Carleman Inequalities for the Heat Equation with Fourier Boundary Conditions: Applications to  Null Controllability Problems}
\author{Jose Antonio Villa }
\date{October 2025}

\begin{document}

\maketitle

\begin{abstract}
   In this work, we establish a Carleman inequality for the heat equation with Fourier boundary conditions of the form $\partial_\nu y+by=f1_\gamma$, where the control acts on a small portion $\gamma$ of the boundary. We apply this inequality to address the null controllability problem with boundary control supported on this small region. An explicit solution to this problem is obtained via a system of coupled parabolic equations. Based on these results, we propose an iterative numerical method to solve the coupled system. 
\end{abstract}


\section{Introduction}

The controllability of parabolic equations has been extensively studied during the last decades, due to its theoretical interest and its relevance in applications. In particular, Carleman inequalities have become a fundamental tool in the analysis of controllability properties for parabolic and hyperbolic equations. These inequalities allow one to derive observability estimates for adjoint systems, which in turn imply controllability results for the original equations. 

Let $\Omega\subset \mathbb{R}^n$ be an open set with boundary $\Gamma$ of class $C^2$. Construct the cylinder $Q=\Omega \times (0,T)$ and the cylinder shell $\Sigma = \Gamma \times (0,T)$. Let $A\in W^{1,\infty}(Q,\mathbb{R}^n)$, $a\in L^\infty(Q)$, $g\in L^2(Q)$, $f\in L^2(\Sigma)$, $b\in L^\infty(\Sigma)$ and $y_0\in L^2(\Omega)$.  Define the linear equation 
\begin{equation}\label{eq0}
    \begin{array}[c]{lll}
        y_t-\Delta y  +A\cdot \nabla y+a y =g & \text{in} & Q, \\
       \partial_\nu y+by= f & \text{on} 
       &\Sigma, \\
       y(0)=y_0  & \text{in} &\Omega.\\
    \end{array}
\end{equation}

This paper is devoted  to deriving  Carleman inequalities localized in a small  control region $\gamma\subset \Gamma$ of the adjoint equation associated to \eqref{eq0} with the initial condition $p_0\in L^2(\Omega)$  defined as
\begin{equation}\label{adjeq0}
    \begin{array}[c]{lll}
        -p_t-\Delta p -\text{div}(Ap)+ap=0& \text{in} &  Q, \\
       \partial_\nu p+pA\cdot\nu+bp=0 & \text{on} &\Sigma, \\
       p(T)=p_0  & \text{in} & \Omega.\\
    \end{array}
\end{equation}
The desired  Carleman inequality is established  in Theorem 1 for a more general adjoint system, which is considered with the initial condition $\psi_0\in L^2(\Omega)$ and is given by 
\begin{equation}\label{adjeq}
    \begin{array}[c]{lll}
        -\psi_t-\Delta \psi =\nabla \cdot F_2 +F_1& \text{in} &  Q, \\
       \partial_\nu \psi+\psi F_2\cdot \nu=F_3 & \text{on} &\Sigma, \\
       \psi(T)=\psi_0  & \text{in} & \Omega,\\
    \end{array}
\end{equation}
  
 where $\nabla \cdot F_2+F_1\in L^2(Q),F_3\in L^2(\Sigma)$. Although our main motivation comes from the adjoint equation associated with \eqref{eq0}, it is convenient to work with the slightly more general adjoint system \eqref{adjeq}.  Let $\eta\in C^\infty(\overline{\Omega})$ and define the weight functions $\alpha$ and $\xi$ by

\begin{equation}\label{weigh}
\alpha(t,x) =  \frac{ e^{2\lambda \|\eta\|_{\infty}} - e^{\lambda \eta(x)} }{t(T-t)}, \hspace{0.5cm}
\xi(t,x) =\frac{ e^{\lambda \eta(x)}}{t(T-t)}, 
\qquad (t,x) \in \overline{Q}.
\end{equation}

In \cite{fernandez2006null}, the authors derive Carleman inequalities for the system \eqref{adjeq} with an interior control  region $\omega\subset \Omega$.  The key idea is that the cut-off  function allows one to control all  boundary terms through the conditions  $\eta|_{\Sigma}=0$, $\eta>0$ on $\Omega$  and $\partial_\nu\eta<-c_0$ on $\Gamma$. As a consequence, boundary integrals  involving $\partial_\nu\eta$ can be controlled by $\omega$. By contrast, in the boundary case considered here, this task becomes more delicate. Indeed, it is necessary  to construct a new function $\eta$ that behaves in a suitable way through on $\Gamma$. This constitutes one of the main technical difficulties that arise when working with boundary conditions.

The main application of  the resulting Carleman inequality is to address the null controllability problem for \eqref{eq0}, where the controls act precisely in the control region $\gamma$ on the boundary. 

\begin{definition}
    Control $f\in L^2(\Sigma)$  solves the null controllability problem with time $T>0$, if the solution $y$ to \eqref{eq0} satisfies
    \begin{equation}\label{null}
        y(T)=0 \,\,\, \text{on}\,\,\, \Omega.
    \end{equation}
  
\end{definition}
The second application is the solution of the semi-linear problem.  Let  $F\in C^1(\mathbb{R})$ be a Lipschitz function.  Define the initial value problem
\begin{equation}\label{eq1}
    \begin{array}[c]{lll}
        y_t-\Delta y +F(y) =g & \text{in} \, Q, \\
       \partial_\nu y+by= f1_\gamma & \text{on} \,\Sigma, \\
       y(0)=y_0  & \text{in} \, \Omega.\\
    \end{array}
\end{equation}

The main results on controllability are stated in Theorem \ref{control1} in the next section, where the existence of the solution  is given for any positive time $T$.  Observe that \eqref{null} can be formulated as a trajectory  controllability problem, by means of a suitable change of variables, in order to obtain $y(T)=\bar{y}$ where $\bar{y}$  is the solution to the homogeneous form of \eqref{eq1} with the initial boundary and the initial condition.

\textbf{Remark:} When nonlinearities of the form $F(y,\nabla y)$ are considered, the linearization process leads to first-order terms in the adjoint equation. The treatment of such terms requires additional structural assumptions and delicate estimates beyond the scope of the present work. For this reason, we restrict our semi-linear analysis to nonlinearities depending only on the state variable. 

The physical context where the Fourier boundary conditions work arise in the study of the heat flux. Let $\Omega\subset \mathbb{R}^n$. Consider the heat flux $q$, and $F$ heat sources. According to  Fourier law \footnote{With $k=b$} $q=-k\nabla y$, it can be deduced that
\begin{equation}\label{foulaw}
    q=-k\nabla y.
\end{equation}
 Suppose that $\Omega$ is a hot body surrounded by a cooler liquid with homogeneous temperature $T_{\infty}$. When $q=0$,  the body is isolated and no heat exchange occurs.  
 \begin{enumerate}
     \item   Thus, condition $\partial_\nu y=-by$ means that the heat that crosses the boundary is proportional to its temperature on $\Gamma$. 
     \item Involving a control $f$ in the small region $\gamma$ has the objective of compensating the heat loss on the boundary,  having an equilibrium via $\partial_\nu y=-by+f1_\gamma$ on $\gamma$. This reflects the idea that the heat loss in the boundary is controlled by the source $f$ in that region.
 \end{enumerate}
 
\subsection{Main results}
We now present the two main theorems of the paper. The first one is  a boundary Carleman inequality for non-homogeneous Fourier boundary conditions. The second one addresses the solution to the null controllability as a consequence of  Theorem 1. To state the first result, we first introduce the following auxiliary lemma.
\begin{lemma}\label{lem1}
Given a non-empty  open set $\gamma \subset \Gamma$, there exists a function $\eta \in C^{2}(\overline{\Omega})$ such that
\begin{align}
\text{(i)}\quad & \eta > 0 && \text{in } \Omega, \\
\text{(ii)}\quad & |\nabla \eta| \ge c_{0} > 0 && \text{in } \overline{\Omega}, \\
\text{(iii)}\quad & \partial_{\nu} \eta \le -c_{0} && \text{on } \Gamma \setminus \gamma, \\
\text{(iv)}\quad & \eta = 0,\ \nabla_{\Gamma} \eta = 0 && \text{on } \Gamma \setminus \gamma.
\end{align}
\end{lemma}
Let $\gamma \subset \Gamma$ be an open and nonempty set, and let $\eta$ be as above.  The weights $\alpha$ and $\xi$ given in \eqref{weigh} are of class $C^{2}$, strictly positive in $\overline{Q}$, and blow-up at $t=0$ and $t=T$. Define the weight functions 
\begin{equation}
    \varrho=e^{s\alpha},\varrho_1=s^{-1/2}\lambda^{-1}\xi^{-1/2}\varrho,\varrho_0=s^{-3/2}\xi^{-3/2}\lambda^{-2}\varrho,\varrho_2=s^{-1}\lambda^{-1}\xi^{-1}\varrho, \varrho_3=s^{-1/2}\lambda^{-1/2}\xi^{-1/2}\varrho.
\end{equation}

\begin{theorem}\label{carleman1}
    Let $F_1,F_2,F_3$ be such that $\nabla F_2+F_1\in L^2(Q)$, $F_3\in L^2(\Sigma)$  and $\psi_0\in L^2(\Omega)$. Let $\psi$ a solution to \eqref{adjeq}. Then there exist $\lambda_*>0$, $s_*>0$ and a constant $C(\Omega,\gamma)$ such that for any $\lambda>\lambda_*$ and $s>s_*$ the next Carleman inequality holds
    \begin{equation}\label{carl2}
    \begin{array}[c]{lll}
       \displaystyle \int_Q\varrho_1^{-2}|\nabla \psi|^2\,dxdt &+ \displaystyle  \int_Q\varrho_1^{-2}|\psi|^2\,dxdt+  \int_\Sigma\varrho_0^{-2}|\psi|^2\,d\Sigma \leq C(\Omega,\gamma)\left(\int_Q\varrho^{-2}|\psi_t+\Delta \psi|\,dxdt \right. \\
         &\displaystyle  +\int_Q\varrho^{-2}|F_1|^2\,dxdt+\int_Q\varrho_2^{-2}|F_2|^2\,dxdt +\int_{\Sigma}\varrho_3^{-2}|F_3|^2\,d\Sigma \\
         &\displaystyle+\int_{\gamma \times (0,T)}\varrho_0^{-2}|\psi|^2\,dxdt\Big). 
    \end{array}
\end{equation}
\end{theorem}
Let $g\in L^2(Q)$ such that
\begin{equation}\label{gcond}
    \varrho_1g\in L^2(Q).
\end{equation}
Define the spaces 
\begin{equation}\label{spac1}
    \mathcal{Y}=\{y:\varrho y\in L^2(Q)\},  \mathcal{F}=\{f:\varrho_0 f\in L^2(\Sigma)\}.
\end{equation}

\begin{theorem}\label{control1}
    Let $T>0$, $g$ with \eqref{gcond}  and $y_0\in L^2(\Omega)$ be given. For $F\in C^2$ Lipschitz with $F(0)=0$, there exists  a $f\in \mathcal{F}$ that solves the null controllability problem \eqref{null} for $y\in \mathcal{Y}$ solving \eqref{eq1}. Moreover,  
    \begin{equation}
        \|f\|_{\mathcal{F}}+\|y\|_{\mathcal{Y}}\leq C(\|g\|_{L^2(Q)}+\|y_0\|_{L^2(\Omega)})
    \end{equation}
\end{theorem}
In Section 2 the null controllability problem is solved for the linear case \eqref{eq0}  and then,  taking $A=0$, we apply a fixed point theorem to deduce the existence of a follower control for the semi-linear case \eqref{eq1}. Via duality method stated in the Dubovitsky-Milyoutin Theorem \ref{dv.yu}  theorem we compute   explicit solutions for  the null controllability problem..

\section{Preliminary}

Given $X$ a Banach space, $s\in \mathbb{R}$, $1\le p<\infty$ we define

$$W^{s,p}(0,T;X)=\left\{f\in L^p(0,T;X) \text{ and } \int_0^T\!\!\!\int_0^T\frac{\|f(t)-f(\tau)\|^p_X}{|t-\tau|^{sp+1}} dtd\tau <\infty\right\}.$$
We recall the following compactness result due to Simon \cite{simon1986compact}, (Corollary 9, p. 90).
 \begin{proposition}\label{5.x.}
        	Let  $X,B$ and $Y$ Banach spaces. Assume that $X \subset B$ is compact and $B\subset Y$. For $s_0,s_1$   reals, $\theta\in (0,1)$ and $1\leq r_0\leq \infty,1\leq r_1\leq \infty$, define the numbers $s_{\theta}=(1-\theta)s_0+\theta s_1$, $ \frac{1}{r_\theta}=\frac{\theta}{r_1}+\frac{1-\theta}{r_0}$ and $s_*=s_{\theta}-\frac{1}{r_\theta}$.  Let $A$ be a bounded set in  $W^{s_0,r_0}(0,T;X)\cap W^{s_1,r_1}(0,T;Y)$. If $s_*\leq 0$ then $A$ is relatively  compact in $L^p(0,T;B)$ for $p< - \frac{1}{s_*}$. 
       \end{proposition}
 
\begin{proposition}[Dubovitski-Milyoutin, \cite{alexeev2017commande}]\label{dv.yu}
      Let $H$ a Hilbert space, $F$  a real function on $H$ and $M\in C^1(H,H)$. Let $\hat{h}\in H$ be a solution to the optimization problem
      \begin{equation}
          F(\hat{h})=\inf_{h\in H} F(h), \hspace{0.5cm}M(\hat{h})=0.
      \end{equation}
      Then, there exist a $\lambda >0$ and $\zeta\in \text{ker}(M'(\hat{h}))^\bot$     not both zero such that 
      \begin{equation}\label{1.3gaw}
         \lambda F'(\hat{h})+\zeta=0.
      \end{equation}
Observe that if $M$ is a compact operator, then $M'$ is compact and moreover $$\ker M'(\hat{h})^\bot=\text{Rank} (M'(\hat{h}))^*.$$    
  \end{proposition}
The proof of the following results can be seen in \cite{LionsMagenes2} page 82. 
   \begin{lemma}\label{regubd}
Let $h\in L^2(Q) $  and $g\in L^2(\Sigma)$. Then there exists a solution $u\in H^{3/2,1/4}(Q)$ to the problem
\begin{equation}
    \begin{array}[c]{lll}
         u_t-\Delta u +au=h&  \hbox{in} &\Omega,\\
         \partial_\nu u+bu=g& \hbox{on}& \Gamma,\\
         u(0)=u_0 &\hbox{on}&\Omega.
         
    \end{array}
\end{equation}
Moreover, it fulfills the inequality
\begin{equation}
    \|u\|_{H^{3/2,1/4}(Q)}\leq C\big(\|h\|_{L^2(\Sigma)}+\|g\|_{L^2(Q)}).
\end{equation}
\end{lemma}
\subsection{Proof of theorem \ref{carleman1}.}
To prove theorem \ref{carleman1} it is necessary to have the next Carleman inequality (\cite{boundnullchorfi},
Lemma 3.2 p. 9 ) for Fourier boundary conditions. 
\begin{lemma}\label{propfi}
 Let $p_0\in L^2(\Omega)$ and let $ p$ such that $p_t+\Delta p\in L^2(Q)$, $\partial_\nu p=0$ and $p(T)=p_0.$  Then there exist positive constants $\lambda_1(\Omega,\gamma)$  and $s_1=C_1(\Omega,\gamma)(T+T^{8/3})$  such that for any $s>s_1$ and $\lambda> \lambda_1$ the next inequality holds,
\begin{equation}\label{carl1}
    \begin{array}{lll}
& \displaystyle  s^{-1}\int_Qe^{-2s\alpha} \xi^{-1}(|p_t|^2+|\Delta p|^2  )dxdt+s\lambda^2\int_Qe^{-2s\alpha}\xi|\nabla p|^2dxdt\\
& +   \displaystyle s\lambda^2\int_Qe^{-2s\alpha}\xi|p|^2dxdt
 \displaystyle +s^3\lambda^3\int_{\Sigma}\!\!\!e^{-2s\alpha}\xi^3|p|^2d\Sigma\\
& \displaystyle  \leq C\left(\int_Qe^{-2s\alpha}|p_t+\Delta p|^2+ s^3\lambda^4\int_{\gamma \times (0,T)}\!\!\!\!\!\!\!\!\! \xi^3e^{-2s\alpha}|p|^2d\Sigma\right).\\
    \end{array}
\end{equation}

\end{lemma}
\textbf{Proof of Theorem \ref{carleman1} :} Consider  $\psi$ the solution to the  adjoint equation \eqref{adjeq}.  For the moment (this will be proved in the next lemma), assume  that there exists $h$ such that  $z\in L^2(0,T;H^1(\Omega))\cap H^1(0,T,L^2(\Omega))$ is the solution  to the equation
\begin{equation}\label{zeq}
    \begin{array}[c]{lll}
        z_t-\Delta z  =0 & \text{in} &\, Q, \\
       \partial_\nu z=h & \text{on}& \,\Sigma, \\
       z(0)=0, z(T)=0 & \text{in} &\, \Omega.\\
    \end{array}
\end{equation}
Multiply \eqref{adjeq} by $z$,  consider $\psi$ as a solution by transposition,  to get
\begin{equation}
    \int_\Sigma \psi h\,dxdt=\int_QF_1z-F_2\cdot\nabla z  \,dxdt+\int_\Sigma F_3 z\,d\Sigma.
\end{equation}

In particular for $h=s^3\lambda^4\xi^3e^{-2s\alpha}\psi+v1_{\gamma}$, the above identity becomes
\begin{equation}\label{432}
    \int_\Sigma s^3\lambda^4\xi^3e^{-2s\alpha}|\psi|^2\,dxdt=\int_Q(F_1z-F_2\nabla z)  \,dxdt+\int_\Sigma F_3 z\,d\Sigma-\int_{\gamma\times(0,T)}\!\!\!\!\!\!\!\!\!\!\psi v\,d\Sigma.
\end{equation}
The strategy to obtain the desired Carleman inequality is to bound each integral of the right hand side of \eqref{432} via Young's inequality with appropriately chosen weights  deduced from Carleman inequality \eqref{carl2}. This leads to a sum of weighted $L^2$ bounds for $z,\nabla z$ and the  trace of $z$. From this sum, the integral  from the left hand side of \eqref{432}  will be controlled  by a $L^2(\gamma\times (0,T))$ integral of $\psi$.  To reach this objective, we require the solution to \eqref{zeq} stated in the next lemma. 

\begin{lemma}
There exists a $v\in L^2(\gamma\times(0,T))$ such that equation \eqref{zeq} has an associated  solution $z\in L^2(Q)$.
\end{lemma}
\begin{lemmaproof}
 Define the optimization problem for the variables $\lambda,s$ chosen as in Lemma \ref{propfi}. 
\begin{equation}\label{zeq1}
    \begin{array}{lll}
    \displaystyle &\underset{v \in L^2(Q)}{\min} \displaystyle\int_{Q}e^{2s\alpha}|z|^2dxdt +s^{-3}\lambda^{-4}\int_{\gamma\times (0,T)}\!\!\!\!\!\xi^{-3}e^{2s\alpha}|v|^2d\Sigma\\
        &z_t-\Delta z  =0 & \text{in} \, Q, \\
       &\partial_\nu z=s^3\lambda^4\xi^3e^{-2s\alpha}\psi+v1_{\gamma} & \text{on} \,\Sigma, \\
       &z(0)=0 & \text{in} \, \Omega.\\
    \end{array}
\end{equation}

Define the vector space $\mathcal{B}_0=\{p\in C^2(\bar{Q}): \partial_\nu p|_{\Sigma}=0 \}$. In $\mathcal{B}_0$ define the bilinear form  
\begin{equation}
    m(0,p,q)=\int_Qe^{-2s\alpha}L_0^*(p)L^*_0(q)\,dxdt+s^3\lambda^4\int_{\gamma\times (0,T)} \!\!\!\xi^4e^{-2s\alpha}pq\,d\Sigma
\end{equation}
and the linear functional
\begin{equation}\label{funl}
    l(p)=s^3\lambda^4\int_{Q}\xi^3e^{-2s\alpha}\psi p\,dxdt.
\end{equation}
By Carleman inequality \eqref{carl1},  $m(0,p,p)=0$ implies $p=0$, so it defines a norm $\|p\|_{\mathcal{B}_0}=m(0;p,p)$ on $\mathcal{B}_0$. By completion, the space  $\mathcal{B}_0$ becomes a Banach space  $(\mathcal{B},\|\cdot \|_{\mathcal{B}})$ and  the functional \eqref{funl} is continuous in this space. By standard optimization arguments \cite{calsavara2022new},\cite{villa2026stackelberg} and \cite{fernandez1997null}, there exists explicit solution for optimization problem  \eqref{zeq1} given by $\hat{p}\in \mathcal{B}$ such that
\begin{equation}\label{inden12}
    m(0,\hat{p},q)=l(q)
\end{equation}

and the solutions  for \eqref{zeq1} are of the form
\begin{equation}
    \hat{z}=e^{-2s\alpha}L^*_0(\hat{p}
    ), \hspace{2cm} \hat{v}=-s^3\lambda^4\xi^3e^{-2s\alpha}\hat{p}1_{\gamma}.
\end{equation}

Denote by $\|\ell\|_{\mathcal{B}}$ the dual norm of the operator. From \eqref{inden12} it is possible to bound
\begin{equation}
    \|\hat{p}\|^2_{\mathcal{B}}\leq \|\ell\|_{\mathcal{B'}}\|\hat{p}\|_{\mathcal{B}}.
\end{equation}
Where, since the choice of $\lambda$ and $s$ , the dual norm is estimated by
\begin{equation}
     \|\ell\|_{\mathcal{B'}}\leq s^{3/2}\lambda^2\left(\int_\Sigma\xi^{3}e^{-2s\alpha}|\psi|^2\,d\Sigma\right)^{1/2}.
\end{equation}

Then we get
\begin{equation}\label{eqf}
    \int_{Q}e^{2s\alpha}|\hat{z}|^2\,dxdt+s^{-3}\lambda^{-4}\int_{\gamma\times (0,T)}\!\!\!\!\!\!\!\!\xi^{-3}e^{2s\alpha}|\hat{v}|^2d\Sigma \leq Cs^3\lambda^4\int_\Sigma\xi^3e^{2s\alpha}|\psi|^2\,dxdt.
\end{equation}
 This lemma allows one to rewrite optimization problem \eqref{zeq1} as a fourth order equation given by
 \begin{equation}
 \begin{array}{lll}
        &L_0(e^{-2s\alpha}L^*_0(\hat{p}
    ))=0 & \text{in} \, Q, \\
       &\partial_\nu(e^{-2s\alpha}L^*_0(\hat{p}
    ))=s^3\lambda^4\xi^3e^{-2s\alpha}\psi-s^3\lambda^4\xi^3e^{2s\alpha}\hat{p}1_{\gamma} & \text{on} \,\Sigma, \\
       &(e^{2s\alpha}L^*_0(\hat{p}
    ))(0)= (e^{2s\alpha}L^*_0(\hat{p}
    ))(T)=0& \text{in} \, \Omega.\\
    \end{array}
\end{equation}
\end{lemmaproof}

\textbf{Step 1:} All the following computations are aimed at deriving the desired weighted $L^2$ estimates for $\hat{z},\nabla \hat{z}$ and the trace of $\hat{z}$. The philosophy behind the structure of the desired weights follows the classical approach introduced  in \cite{fernandez2006null}  and \cite{emanuilov1995controllability}.  Consequently,  throughout the forthcoming analysis we expect to get weights of the form $e^{2s\alpha}\xi^{-m}$ with $m\geq0$. Multiply  equation   \eqref{zeq} by $s^{-2}\lambda^{-2}\xi^{-2}e^{2s\alpha}\hat{z}$, integrate by parts to get
\begin{equation}\label{32.9}
\begin{array}[c]{lll}
\displaystyle \frac{1}{2}s^{-2}\lambda^{-2}\int_{Q}\xi^{-2}e^{2s\alpha}\frac{d}{dt}|\hat{z}|^2\,dxdt+s^{-2}\lambda^{-2}\int_Q \xi^{-2}e^{2s\alpha}|\nabla \hat{z}|^2\,dxdt\\
+\displaystyle s^{-2}\lambda^{-2}\int_Q \nabla(\xi^{-2}e^{2s\alpha})\hat{z}\cdot \nabla \hat{z}\,dxdt = s\lambda ^2\int_\Sigma \xi\psi \hat{z}\,d\Sigma+s^{-2}\lambda^{-2}\int_{\gamma\times (0, T)}\!\!\!\!\!\!\!\!\!\!\!\xi^{-2} e^{2s\alpha}v\hat{z}\,d\Sigma
\end{array}
\end{equation}
Now we compute the gradient in the third integral of the left hand side of \eqref{32.9} to get
\begin{equation}\label{32.10}
\begin{array}[c]{lll}
\displaystyle \frac{1}{2}s^{-2}\lambda^{-2}\int_{Q}\xi^{-2}e^{2s\alpha}\frac{d}{dt}|z|^2\,dxdt+s^{-2}\lambda^{-2}\int_Q \xi^{-2}e^{2s\alpha}|\nabla z|^2\,dxdt+\\
\displaystyle s^{-2}\lambda^{-2}\int_Q(-2\lambda \xi^{-2}e^{2s\alpha}  - 2s\xi^{-1}e^{2s\alpha})\nabla \eta \cdot z\nabla z\,dxdt = s\lambda ^2\!\!\!\int_\Sigma \xi\psi z\,d\Sigma\\
\displaystyle +s^{-2}\lambda^{-2}\int_{\gamma\times (0, T)}\!\!\!\!\!\!\!\!\!\!\!\xi^{-2}e^{2s\alpha}v\hat{z}\,d\Sigma.
\end{array}
\end{equation}
Integrate by parts the first integral and move it to the right hand side,  take $\hat{z}\nabla \hat{z}=\frac{1}{2}\nabla |\hat{z}|^2$. Also $2\xi^{-1}e^{2s\alpha}\nabla \eta \cdot (\hat{z}\nabla \hat{z})=\text{div}( |\hat{z}|^2\xi^{-1}e^{2s\alpha}\nabla \eta)-\nabla\cdot(\xi^{-1}e^{2s\alpha}\nabla \eta)\cdot |\hat{z}|^2$. Finally, apply the Stokes theorem and move the resulting boundary integral to the left hand side to get
\begin{equation}\label{32d}
\begin{array}[c]{lll}
\displaystyle s^{-2}\lambda^{-2}\!\!\!\int_Q \xi^{-2}e^{2s\alpha}|\nabla \hat{z}|^2\,dxdt-s^{-1}\lambda^{-2}\!\!\!\!\int_{\Sigma}\xi^{-1}e^{2s\alpha}\partial_{\eta}\eta|\hat{z}|^2\,d\Sigma = \frac{1}{2}s^{-2}\lambda^{-2}\!\!\!\int_{Q}\frac{d}{dt}(\xi^{-2}e^{2s\alpha})|\hat{z}|^2\,dxdt\\
-s^{-1}\lambda^{-1}\displaystyle\int_Q \nabla \cdot(\xi^{-1}e^{2s\alpha}\nabla \eta)\cdot |\hat{z}|^2\,dxdt + 2s^{-2}\lambda^{-1}\int_Q  \xi^{-2}e^{2s\alpha}\nabla \eta \cdot \hat{z}\nabla \hat{z}\,dxdt\\
+\displaystyle s\lambda ^2\!\!\!\int_\Sigma \xi\psi \hat{z}\,d\Sigma+s^{-2}\lambda^{-2}\int_{\gamma\times (0, T)}\!\!\!\!\!\!\xi^{-2}e^{2s\alpha}v\hat{z}\,d\Sigma.
 \end{array}
\end{equation}

The next lemma can be found in \cite{fernandez2006null}, but for completeness, its proof is given.
\begin{lemma}
    For each $m$ real number,  the following inequalities hold
    \begin{equation}\label{12.4}
      |\nabla(e^{2s\alpha}\xi^m)|<C(\Omega,\gamma)s\lambda e^{2s\alpha}\xi^{m+1}
\end{equation}
and
\begin{equation}\label{567}
   \left| \frac{d}{dt}(\xi^{k}e^{2s\alpha})\right|<C(\Omega,\gamma)Ts\xi^{k+1}e^{2s\alpha}e^{2\lambda\|\eta\|_{\infty}}.
\end{equation}

\end{lemma}
\begin{proof}
Let 
\begin{equation}
    \begin{array}[c]{lll}
    |\nabla(e^{2s\alpha}\xi^m)|=|e^{2s\alpha}\lambda \xi^m\nabla \eta_0(2s\xi+m)|<Ce^{2s\alpha}\lambda \xi^m(s\xi+1).
    \end{array}
\end{equation}
Then for $s>\frac{T^2}{4C}$ it holds $C(\Omega,\gamma)s\xi>1$ and then
\begin{equation}\label{45a}
    s^{-1}\xi^{-1}<C(\Omega,\gamma).
\end{equation}
Invoking this in the above inequality, we get
\begin{equation}\label{12.42}
      |\nabla(e^{2s\alpha}\xi^m)|<C(\Omega,\gamma)s\lambda e^{2s\alpha}\xi^{m+1}.
\end{equation}
For the second inequality, proceed in the same way.
\end{proof}

\textbf{Step 2:} Proceed now to estimate each of the integrals on the right hand side of \eqref{32d}. The idea is to separate each integral into weighted $L^2$integrals depending only on $|\hat{z}|$ or $|\nabla \hat{z}|^2$. Go ahead with the first one using the condition \eqref{567} to get
\begin{equation}\label{inesd1}
    \frac{1}{2}s^{-2}\lambda^{-2}\!\!\!\int_{Q}\frac{d}{dt}(\xi^{-2}e^{2s\alpha})|\hat{z}|^2\,dxdt\leq C(\Omega,\gamma)s^{-1}\lambda^{-2}e^{2\lambda\|\eta\|_{\infty}}\int_Q\xi^{-1}e^{2s\alpha}|\hat{z}|^2\,dxdt.
\end{equation}
 Proceed now to evaluate the second integral on the right hand side of \eqref{32d}. Estimate the gradient $\nabla(\xi^{-5/2}e^{2s\alpha}\nabla \eta)$ in two stages: first, by lemma \ref{lem1} we find that the Laplacian $\|-\Delta \eta\|_\infty<C_0$ is bounded on $\bar{\Omega}$ and  second,  use \eqref{12.4} to get
 \begin{equation}\label{inesd2}
     -s^{-1}\lambda^{-1}\displaystyle\int_Q \nabla \cdot(\xi^{-1}e^{2s\alpha}\nabla \eta)\cdot |\hat{z}|^2\,dxdt\leq C(\Omega,\gamma)\|\nabla \eta\|_{\infty}\lambda^{-1}\int_Q e^{2s\alpha}|\hat{z}|^2\,dxdt+s^{-1}\lambda^{-1}C_0\int_Q\xi^{-1}e^{2s\alpha}|\hat{z}|^2\,dxdt.
 \end{equation}
 
In the third integral, proceed through the Young inequality to obtain
\begin{equation}\label{inesd3}
\begin{array}[c]{llll}
    \displaystyle 2s^{-2}\lambda^{-1}\int_Q  \xi^{-2}e^{2s\alpha}\nabla \eta \cdot \hat{z}\nabla \hat{z}\,dxdt\leq  2s^{-2}\lambda^{-1}\|\nabla \eta\|_{\infty}\int_Q  \xi^{-2}e^{2s\alpha} \hat{z}|\nabla \hat{z}|\,dxdt\\
    \displaystyle \leq \frac{1}{2} s^{-2}\lambda^{-1}\int_Q\xi^{-2}e^{2s\alpha}|\nabla \hat{z}|^2\,dxdt+2s^{-2}\lambda^{-1}\int_Q\xi^{-2}e^{2s\alpha}| \hat{z}|^2\,dxdt.
    \end{array}
\end{equation}
Finally, estimate the integrals involving $\psi$ and $v$ 
\begin{equation}\label{inesd4}
\begin{array}[c]{lll}
   \displaystyle  s\lambda ^2\!\!\!\int_\Sigma \xi\psi \hat{z}\,d\Sigma &\displaystyle-s^{-2}\lambda^{-2}\int_{\gamma\times (0, T)}\!\!\!\!\!\!\xi^{-2}e^{2s\alpha}v\hat{z}\,d\Sigma \leq  s^{-4}\lambda^{-4}\int_{\gamma\times(0,T)}\xi^{-4}e^{2s\alpha} |v|^2\,d\Sigma \\
   &\displaystyle+\lambda^{-1}\int_{\gamma\times (0,T)}e^{2s\alpha}|\hat{z}|^2\,d\Sigma +s^2\lambda^4\int_\Sigma\xi^2e^{-2s\alpha}|\psi|^2\,dxdt + \int_\Sigma e^{2s\alpha}|\hat{z}|^2\,dxdt.
     \end{array}
\end{equation} 

Combining \eqref{inesd1},\eqref{inesd2},\eqref{inesd3}, \eqref{inesd4}  in \eqref{32d}, one gets
\begin{equation}\label{neq45}
\begin{array}[c]{lll}
   \displaystyle s^{-2}\lambda^{-2}\!\!\!\int_Q \xi^{-2}e^{2s\alpha}|\nabla \hat{z}|^2\,dxdt&\displaystyle -s^{-1}\lambda^{-2}\!\!\!\!\int_{\Sigma}\xi^{-1}e^{2s\alpha}\partial_{\eta}\eta|\hat{z}|^2\,d\Sigma\\
    &\displaystyle\leq C(\Omega,\gamma)s^{-1}\lambda^{-2}e^{2\lambda\|\eta\|_{\infty}}\int_Q\xi^{-1}e^{2s\alpha}|\hat{z}|^2\,dxdt \\ 
    &+\displaystyle C(\Omega,\gamma)\|\nabla \eta\|_{\infty}\lambda^{-1}\int_Q e^{2s\alpha}|\hat{z}|^2\,dxdt+s^{-1}\lambda^{-1}C_0\int_Q\xi^{-1}e^{2s\alpha}|\hat{z}|^2\,dxdt\\
    & +\displaystyle  \frac{1}{2}s^{-2}\lambda^{-2}\int_Q\xi^{-2}e^{2s\alpha}|\nabla \hat{z}|^2\,dxdt+s^{-2}\lambda^{-1}\int_Q\xi^{-2}e^{2s\alpha}| \hat{z}|^2\,dxdt\\
    &+\displaystyle s^{-4}\lambda^{-4}\int_{\gamma\times(0,T)}\xi^{-4}e^{2s\alpha} |v|^2\,d\Sigma \\
   &\displaystyle+\lambda^{-1}\int_{\gamma\times (0,T)}e^{2s\alpha}|\hat{z}|^2\,d\Sigma ++s^2\lambda^4\int_\Sigma\xi^2e^{-2s\alpha}|\psi|^2\,dxdt + \int_\Sigma e^{2s\alpha}|\hat{z}|^2\,dxdt.
    \end{array}
\end{equation}
It remains to bound the weight functions by a constant. We know that $\lambda^{-1}\leq \lambda_*^{-1}$, and due to property \eqref{45a}, it is possible to absorb  the powers of $\lambda^{-1}$ in the constant $C(\Omega,\gamma)$. Under this argument, \eqref{neq45} can be written as
\begin{equation}\label{5672}
\begin{array}[c]{lll}
\displaystyle s^{-2}\lambda^{-2}\int_Q \xi^{-2}e^{2s\alpha}|\nabla \hat{z}|^2\,dxdt-s^{-1}\lambda^{-1}\int_{\Sigma}\xi^{-1}e^{2s\alpha}\partial_\nu\eta|\hat{z}|^2\,d\Sigma \leq C(\Omega,\gamma)\left(s^{-1}\lambda^{-1}\int_Q e^{2s\alpha}|\hat{z}|^2\,dxdt\right. \\
\displaystyle  +\lambda^{-1}\int_Qe^{2s\alpha}|\hat{z}|^2\,dxdt+2\lambda^{-1}\!\!\!\!\int_\Sigma e^{2s\alpha}| \hat{z}|^2\,dxdt+\left.s^3\lambda^4\!\!\!\int_\Sigma\xi^3 e^{2s\alpha}|\psi|^2\,dxdt+s^{-3}\lambda^{-4}\!\!\!\!\int_{\gamma\times(0,T)}\!\!\!\!\!\!\!\!\!\!\xi^{-3}e^{2s\alpha}|v|^2\,dxdt\right)\\
+\displaystyle \frac{1}{2} s^{-2}\lambda^{-2}\int_Q\xi^{-2}e^{2s\alpha}|\nabla \hat{z}|^2\,dxdt.
 \end{array}
\end{equation}
 Now we need to control the second integral of the above inequality by a boundary integral free from the normal derivative $\partial_\nu\eta$. Split the boundary $\Gamma =\Gamma/\gamma \cup \gamma$ and write
\begin{equation}\label{234}
-s^{-1}\lambda^{-1}\!\!\!\!\int_{\Sigma}\xi^{-1}e^{2s\alpha}\partial_{\eta}\eta|\hat{z}|^2\,d \Sigma= -s^{-1}\lambda^{-1}\!\!\int_{\gamma \times (0,T)}\!\!\!\!\!\!\!\!\!\!\!\!\xi^{-1}e^{2s\alpha}\partial_{\eta}\eta|\hat{z}|^2\,d\Sigma -s^{-1}\lambda^{-1}\!\!\int_{\Gamma/\gamma\times(0,T)}\!\!\!\!\!\!\!\!\!\!\!\!\!\!\!\!\xi^{-1}e^{2s\alpha}\partial_{\eta}\eta|\hat{z}|^2\,d \Sigma.
\end{equation}
From \ref{lem1} we have that $-\partial_\nu\eta \geq c_0$ on $\Gamma/\gamma \times(0,T)$ then the second integral on the right hand side of the above equality we can bound
\begin{equation}
  c_0  s^{-1}\lambda^{-1}\!\!\int_{\Gamma/\gamma\times(0,T)}\!\!\!\!\!\!\!\!\!\!\!\!\!\!\!\!\xi^{-1}e^{2s\alpha}|\hat{z}|^2\,d\Sigma \leq   -s^{-1}\lambda^{-1}\!\!\int_{\Gamma/\gamma\times(0,T)}\!\!\!\!\!\!\!\!\!\!\!\!\!\!\!\!\xi^{-1}e^{2s\alpha}\partial_{\eta}\eta|\hat{z}|^2\,d\Sigma.
\end{equation}
It remains to bound the integral on $\gamma\times (0,T)$: apply the trace theorem in the last integral to get 
\begin{equation}
    s^{-1}\lambda^{-1}\int_{\gamma\times(0,T)} \xi^{-1}e^{2s\alpha}|\hat{z}|^2\,d\Sigma\leq    s^{-1}\lambda^{-2}\int_0^T\|\xi^{-1}e^{s\alpha}\hat{z}\|^2_{H^1(\Omega)}dt.
\end{equation}
Estimate the right hand side by inequality \eqref{12.4} and get
\begin{equation} \label{bdry12}
     s^{-1}\lambda^{-1}\int_{\gamma\times(0,T)} \xi^{-1}e^{2s\alpha}|\hat{z}|^2\,d\Sigma \leq   s^{-1}\lambda^{-2}\int_Q\xi^{-1}e^{2s\alpha}|\hat{z}|^2\,dxdt+s^{-1}\int_Qe^{2s\alpha}|\nabla \hat{z}|^2\,d\Sigma.
\end{equation}
In \eqref{5672}, apply the decomposition \eqref{234} and translate the integral on $\gamma$ to the right hand side and use the above inequality to get 
\begin{equation}\label{56721}
\begin{array}[c]{lll}
\displaystyle s^{-2}\lambda^{-2}\!\!\!\int_Q \xi^{-2}e^{2s\alpha}|\nabla \hat{z}|^2\,dxdt+s^{-1}\lambda^{-1}\!\!\!\!\int_{\Sigma}\xi^{-1}e^{2s\alpha}|\hat{z}|^2\,d\Sigma \leq C(\Omega,\gamma)\left(s^{-1}\lambda^{-1}\int_Q e^{2s\alpha}|\hat{z}|^2\,dxdt\right. \\
\displaystyle  +s^{-1}\lambda^{-1}\int_Q\xi^{-1}e^{2s\alpha}|\hat{z}|^2\,dxdt+s^2\lambda^4\int_Q\xi^{2}e^{2s\alpha}| \hat{z}|^2\,dxdt\\
\displaystyle+\left.s\lambda^2\int_\Sigma\xi e^{2s\alpha}|\psi|^2\,dxdt+s^{-2}\lambda^{-4}\int_{\gamma\times(0,T)}\!\!\!\!\!\!\!\!\!\!\xi^{-4}e^{2s\alpha}|v|^2\,dxdt\right).
 \end{array}
\end{equation}

Finally, use inequality \eqref{bdry12} in \eqref{5672} to get \footnote{The constant $C(\Omega,\gamma)$ is not the same line by line; it absorbs different terms along the argument, but to avoid overwhelming notation, we use the same symbol through.}
\begin{equation}\label{12v}
\begin{array}[c]{lll}
    \displaystyle s^{-2}\lambda^{-2}\!\!\!\int_Q \xi^{-2}e^{2s\alpha}|\nabla \hat{z}|^2\,dxdt+s^{-1}\lambda^{-1}\!\!\!\!\int_{\Sigma}\xi^{-1}e^{2s\alpha}|\hat{z}|^2\,d\Sigma \leq\\
    \displaystyle C(\Omega,\gamma)\left(\int_Qe^{2s\alpha}|\hat{z}|^2\,dxdt+s^3\lambda^4\int_\Sigma\xi^3e^{-2s\alpha}|\psi|^2\,dxdt+s^{-3}\lambda^{-4}\int_{\gamma\times(0,T)}\!\!\!\!\!\!\!\!\!\!\xi^{-3}e^{2s\alpha}|v|^2\,dxdt \right).
    \end{array}
\end{equation}
where $\lambda >C(\Omega,\gamma)$ and $s\geq C(\Omega,\gamma)(e^{2\|\eta\|_{\infty}}T+T^2)$. Add the left-hand side of the inequality \eqref{eqf} to both sides \eqref{12v} to get
\begin{equation}\label{13v}
\begin{array}[c]{lll}
    \displaystyle s^{-2}\lambda^{-2}\!\!\!\int_Q \xi^{-2}e^{2s\alpha}|\nabla \hat{z}|^2\,dxdt &\displaystyle+s^{-1}\lambda^{-1}\!\!\!\!\int_{\Sigma}\xi^{-1}e^{2s\alpha}|\hat{z}|^2\,d\Sigma\\
  &+ \displaystyle   \int_{Q}e^{2s\alpha}|\hat{z}|^2\,dxdt+s^{3}\lambda^{4}\int_{\gamma\times (0,T)}\!\!\!\!\!\!\!\!\xi^{3}e^{2s\alpha}|v|^2d\Sigma \leq\\
&\displaystyle\leq C(\Omega,\gamma)s^3\lambda^4\int_\Sigma\xi^3e^{-2s\alpha}|\psi|^2\,dxdt.
    \end{array}
\end{equation}
Now return to the identity \eqref{432}. Apply Young inequality (considering the weights above) to get
\begin{equation}\label{bint1}
    \begin{array}[c]{lll}
     & \displaystyle  s^3\lambda^4\int_\Sigma\xi^3e^{-2s\alpha}|\psi|^2\,dxdt \leq \int_Qe^{-2s\alpha}|F_1|^2\,dxdt  +s^2\lambda^2\!\!\!\int_Q \xi^2e^{-2s\alpha}|F_2|^2\,dxdt\\
         &\displaystyle  +s\lambda\int_{\Sigma}\xi e^{-2s\alpha}|F_3|^2\,d\Sigma +s^3\lambda^4\int_{\gamma \times (0,T)}\xi^3e^{-2s\alpha}|\psi|^2\,dxdt. 
    \end{array}
\end{equation}
\textbf{Step 4:} The second part of the proof relies on establishing estimates for the gradient $|\nabla \psi|^2$  on $Q$. To this end, it is necessary to work in the weak  sense of the solution of \eqref{432}. Let $F_1\in L^2Q)$, $F_2\in L^2(Q)^n$ and $F_3\in L^2(\Sigma)$. A function $\psi\in L^2(0,T;H^1(\Omega))\cap C^0(0,T,L^2(\Omega))$ is a weak solution to \eqref{432}
if  
\begin{equation}
    -\langle \psi_t,u\rangle+\int_\Omega\nabla \psi \cdot \nabla u\,dxdt=\int_\Omega (F_1 u-F_2\cdot\nabla u)\,dx+\int_{\partial\Omega}\!\!\!\!F_3u\,d\Sigma, \hspace{0.4cm} \forall u\in H^1(\Omega).
\end{equation}
Take the function 
\begin{equation}
        u=s\lambda^2\xi e^{-2s\alpha}\psi.
\end{equation}
Observe  that 

\begin{equation}\label{hng}
    \begin{array}[c]{lll}
\displaystyle s\lambda^2\int_Q\xi e^{-2s\alpha}|\nabla \psi|^2\,dxdt&=\displaystyle s\lambda^2\int_Q(\xi e^{-2s\alpha})_t|\psi|^2\,dxdt +s\lambda^2\int_Q\psi \nabla(\xi e^{-2s\alpha} )\cdot\nabla \psi\,dxdt \\
&\displaystyle+s\lambda^2\int_Q \xi e^{-2s\alpha}F_1\psi-s\lambda^2\int_QF_2\nabla(\xi e^{-2s\alpha}\psi)\,dxdt\\
&\displaystyle+s\lambda^2\int_\Sigma\xi e^{-2s\alpha}F_3\psi\,d\Sigma.
    \end{array}
\end{equation}
Then
\begin{equation}\label{y1}
    s\lambda^2\int_Q \xi e^{-2s\alpha}F_1\psi\,dxdt\leq C(\Omega,\gamma)\left(\int_Qe^{-2s\alpha}|F_1|^2\,dxdt+s\lambda^2\int_Q\xi^2e^{-2s\alpha}|\psi|^2\,dxdt\right)
\end{equation}

Then by property \eqref{12.4} and Young inequality to split on $F_2,\psi$ and $\nabla \psi$ to get
\begin{equation}\label{y2}
\begin{array}[c]{lll}
  \displaystyle -s\lambda^2  \int_QF_2\nabla(\xi^2e^{-2s\alpha}\psi)\,dxdt&\leq \displaystyle C(\Omega,\gamma)\left(s^2\lambda^2\int_Q\xi^2 e^{-2s\alpha}|F_2|^2\,dxdt \right.\\
  &\displaystyle \left. +s^3\lambda^4\int_Q\xi^3e^{-2s\alpha}|\psi|^2\,dxdt\right)+\frac{1}{2}s\lambda^2\int_Q\xi e^{-2s\alpha}|\nabla \psi|^2\,dxdt.
  \end{array}
\end{equation}
Now
\begin{equation}\label{y3}
    s\lambda^2\int_\Sigma\xi e^{-2s\alpha}F_3\psi\,d\Sigma \leq C(\Omega,\gamma)\left(s\lambda\int_{\Sigma}\xi e^{-2s\alpha}|F_3|^2\,d\Sigma  +\int_{\Sigma}|\psi|^2\,d\Sigma  \right)
\end{equation}
Use \eqref{y1}, \eqref{y2} and \eqref{y3} in \eqref{hng},  to get
\begin{equation}
    \begin{array}[c]{lll}
    \displaystyle   s\lambda^2\int_Q\xi e^{-2s\alpha}|\nabla \psi|^2\,dxdt   \displaystyle & \displaystyle\leq C(\Omega,\gamma) \Big(\int_Qe^{-2s\alpha}|F_1|^2\,dxdt+s\lambda^2\int_Q\xi e^{-2s\alpha}|\psi|^2\,dxdt\\
          &+\displaystyle s\lambda^2\int_Q\xi e^{-2s\alpha}|F_2|^2\,dxdt+s^3\lambda^4\int_Q\xi^3e^{-2s\alpha}|\psi|^2\,dxdt\\
         &+\displaystyle s\lambda\int_\Sigma \xi e^{-2s\alpha}|F_3|^2\,d\Sigma +s\lambda^3\int_\Sigma \xi e^{-2s\alpha}|\psi|^2\,d\Sigma \Big).
         \end{array}
\end{equation}
Using inequality \eqref{bint1} and the fact that $s\lambda^3\xi\leq s^3\lambda^3\xi^3$, we get
\begin{equation}\label{bint2}
    \begin{array}[c]{lll}
     & \displaystyle   s\lambda^2\int_Q\xi e^{-2s\alpha}|\nabla \psi|^2\,dxdt\leq \int_Qe^{-2s\alpha}|F_1|^2\,dxdt  +s^2\lambda^2\!\!\!\int_Q \xi^2e^{-2s\alpha}|F_2|^2\,dxdt\\
         &\displaystyle  +s\lambda\int_{\Sigma}\xi e^{-2s\alpha}|F_3|^2\,d\Sigma +s^3\lambda^4\int_{\gamma \times (0,T)}\xi^3e^{-2s\alpha}|\psi|^2\,dxdt 
    \end{array}
\end{equation}
Now, remains to compute the $L^2(Q)$ estimates for $|\psi|^2$.  Since the weight function is strictly positive and smooth in $\overline{\Omega}$, a weighted Poincaré inequality holds and it is possible to get
\begin{equation}\label{bint3}
    s\lambda^2\int_Q\xi e^{-2s\alpha}|\psi|^2\,dxdt\leq C s\lambda^2\int_Q\xi e^{-2s\alpha}|\nabla \psi|^2\,dxdt.
\end{equation}
Combining \eqref{bint1}, \eqref{bint2} and \eqref{bint3} it is possible to get \eqref{carl1}
. 

\section{Null controllability problem}
We solve the linear case for \eqref{eq1}. Consider
\begin{equation}\label{eq2}
    \begin{array}[c]{lll}
        y_t-\Delta y +A\cdot \nabla y+ay =g & \text{in} \, Q, \\
       \partial_\nu y+by=  f1_\gamma & \text{on} \,\Sigma, \\
       y(0)=y_0  & \text{in} \, \Omega.\\
    \end{array}
\end{equation}
The key idea to solve null control problem lies in optimization arguments.  Recall that the weights involved in the Carleman inequality blow-up when $t\rightarrow T$, so we can ensure that if the control $f$ makes that the solution $y$ of \eqref{eq2} satisfies $y\in \mathcal{Y}$,  then the solution will reach the condition \eqref{null}.  

\begin{definition}
Define the functional $S:\mathcal{F}\rightarrow \mathbb{R}$ by 
    \begin{equation}
    S(f)=\frac{1}{2}\int_{Q}\varrho^2|y|^2\,dxdt+\frac{1}{2}\int_{\gamma\times (0,T)}\varrho_0^2|f|^2\,d\Sigma.
\end{equation}
\end{definition}

subject to \eqref{eq1}, where $\varrho_0$ and $\varrho$ are  functions such that blow up when $t\rightarrow T$. The existence of a minimizer of this functional and the blowing-up condition on the weights will steer that $y$ solves  \eqref{null} . This approach was used in \cite{calsavara2022new} and \cite{villa2026stackelberg}.  
Observe that $\mathcal{Y}\subset L^2(Q)$ and $\mathcal{F}\subset L^2(\Sigma)$. The control $f$ fulfills \eqref{null} as a  consequence of the optimization problem 
\begin{equation}\label{opnull}
   S(f)=\inf_{\tilde{f}\in \mathcal{F}}S(\tilde{f})  
\end{equation}

For the rest of the text, denote by $L_{A,a}=\partial_t-\Delta -\text{div}(A\cdot)+a$ and $L_0$ the trivial case. Let $\nu\in \Gamma(\Sigma)$ the normal vector field on the boundary of $\Omega$ and $A\in W^{1,\infty}(Q,\mathbb{R}^n)$ . Define the set
\begin{equation}\label{defP}
    \mathcal{P}_0=\{p\in C^2(\bar{Q}):\partial_\nu p+p(b+A\cdot \nu)=0\}
\end{equation}
In $\mathcal{P}_0$, define the bilinear form 
\begin{equation}\label{bilin}
    B(A,a;p,q)=\int_{Q}\varrho^{-2}L_{A,a}^*(p)L_{A,a}^*(q)+\int_{\gamma\times(0,T)}\varrho_0^{-2}pq\,d\Sigma
\end{equation}
Define the semi-norm $\|\cdot\|_{\mathcal{P}_0}=B(0,0;p,p)^{1/2}$. By the Carleman inequality  stated in Theorem\ref{carleman1} the semi-norm is not degenerate; hence it defines a norm.

\begin{definition}
    Define the normed space $(\mathcal{P},\|\cdot\|_{\mathcal{P}})$ as the completion of the semi-normed space $(\mathcal{P}_0,\|\cdot\|_{\mathcal{P}_0})$. 

\end{definition}
The completion ensures that the Carleman inequality in Theorem \ref{carleman1} holds on $\mathcal{P}$ and that the boundary constraint in \eqref{defP} holds on $\mathcal{P}$.
\begin{proposition}\label{prop1}
    Let $T>0$, $y_0\in L^2(\Omega)$ and $g\in L^2(Q)$ with \eqref{gcond}. Then there exists a follower control $f\in \mathcal{F}$ such that the solution $y\in \mathcal{Y}$ satisfies \eqref{null}. Moreover, we have
    \begin{equation}\label{nullsyst}
    \begin{array}[c]{lll}
        y_t-\Delta y +A\cdot \nabla y+ay =g & \text{in} \, Q, \\
       \partial_\nu y+by= \varrho_0^{-2}p1_{\gamma} & \text{on} \,\Sigma, \\
       y(0)=y_0  & \text{in} \, \Omega.\\
    \end{array}
\end{equation}
with
\begin{equation}
    y=\varrho^{-2}L_{A,a}^*(p).
\end{equation}

where $p\in \mathcal{P}$ solves 
\begin{equation}
   B(A,a;p,q)=\langle g,p\rangle_{L^2(Q)}+\langle y_0,p(0)\rangle_{L^2(\Omega)},\,\,\,\,\,\,\forall q\in \mathcal{P}.
\end{equation}
Also
    \begin{equation}\label{ioneq12}
        \|y\|_{\mathcal{Y}} +\|f\|_{\mathcal{F}}< C(\|g\|+\|y_0\|_{L^2(\Omega)})
    \end{equation}

\end{proposition}
\textit{Proof:} Consider the optimal problem \eqref{opnull}. The functional $S:\mathcal{F}\rightarrow \mathbb{R}$ is convex, coercive and lower semicontinuous. Therefore, it admits  a minimizer  $f\in \mathcal{F}$. Compute the Gateaux derivate and  for any $h\in\mathcal{F}$ 
\begin{equation}
    \int_{\gamma \times (0,T)}  \varrho_0^{2}(p+h)\,d\Sigma=0
\end{equation}
where it follows that 
\begin{equation}
    f=\varrho_0^{-2}p1_{\gamma}, \hspace{0.5cm}y=\varrho^{-2}L^*_{A,a}(p).
\end{equation}
Multiply equation \eqref{eq2} by $q\in \mathcal{P}$ to get the equality
\begin{equation}
  B(A,a;p,q)=\int_{Q}gq\,d\Sigma+\langle y_0,q(0)\rangle_{L^2(\Omega)} \hspace{0.5cm} \forall q\in \mathcal{P}.
\end{equation}
The above equation has one solution $p\in \mathcal{P}$ since the coercivity of $B(A,a;\cdot,\cdot)$ and the continuity of the operator
\begin{equation}
    l(q)= \int_{Q}gq\,d\Sigma+\langle y_0,q(0)\rangle_{L^2(\Omega)}.
\end{equation}

By the Hölder inequality it is possible to get inequality \eqref{ioneq12}.
\subsection{The semi-linear case}
Let $F$ be a Lipschitz function in $C^2(\mathbb{R})$. Observe that it is possible to linearize $F$ in the form
\begin{equation}\label{Flinar}
    F(x)=F_0(x)x
\end{equation}
where 
\begin{equation}\label{linl1}
    F_0(x)=
    \begin{cases}
       \displaystyle\frac{F(s)}{s} & s \neq 0\\
        0&s=0.
    \end{cases}
\end{equation}
Define
\begin{equation}
    L_0(p)=p_t-\Delta p
\end{equation}
and
\begin{equation}
    L_a:=L_{0,a}.
\end{equation}

By the global Lipschitz assumption, we can see that $F_0\in L^\infty(Q)$. The idea of the proof is to linearize equation \eqref{12.4}, find controls via Proposition \ref{prop1} and apply Schauder's fixed point theorem to find explicit solutions.  This is addressed in the next Proposition. In this case $A=0$, so the space $\mathcal{P}$ is considered with the boundary condition $\partial_\nu p+bp=0.$

\begin{proposition}\label{nullsemi}
    Let $T>0$, $y_0\in L^2(\Omega)$  and $g$ with \eqref{gcond} be given. There exists a  control $f\in \mathcal{F} $ such that
    \begin{equation}\label{eq11}
    \begin{array}[c]{llll}
        y_t-\Delta y +F(y) =g & \text{in} \, Q, \\
       \partial_\nu y+by= f1_\gamma & \text{on} \,\Sigma, \\
       y(0)=y_0  & \text{in} \, \Omega,\\
    \end{array}
\end{equation}
    with 
    \begin{equation}\label{sol123}
        \displaystyle f =-\varrho_0^{-2}p1_\gamma, \hspace{1cm} y=\varrho^{-2}L_{F'(y)}p,
    \end{equation}

and $p\in \mathcal{P}$ solves 
\begin{equation}\label{eqp}
    \begin{array}[c]{lll}
        	 \displaystyle \int_Q\varrho^{-2}L^*_{F'(y)}(p)L_{F_0(y)}^*(q)&\displaystyle+\int_{\gamma\times(0,T)}\!\!\!\!\!\!\!\!\!\!\!\!\varrho_0^{-2} p\, q\,d\Sigma= \int_Q gp\,dxdt+\int_\Omega y_0q(0)dx,    \forall q\in \mathcal{P} .\\
       
               \end{array}
		\end{equation}
\end{proposition}

\begin{proof}
    Let $z\in L^2(Q)$. Recall that  $F_0(z)\in L^\infty(Q)$. Linearize \eqref{eq1} using the identity \eqref{Flinar} for the function $F$  and invoke Proposition \ref{prop1} to see that there exists a follower control $f_z\in \mathcal{F}$ such that 
 \begin{equation}\label{eq1lin}
    \begin{array}{lll}
        \partial_ty_z-\Delta y_z +F_0(z)y_z =0 & \text{in} \, Q, \\
       \partial_\nu y_z+by_z= f_z1_\gamma & \text{on} \,\Sigma, \\
       y_z(0)=y_0, y_z(T)=0  & \text{in} \, \Omega.\\
    \end{array}
\end{equation}
 
  Define the map $\Lambda: L^2(Q)\rightarrow L^2(Q),\,\Lambda(z)=y_z$. By the inequality of Proposition \eqref{regubd} we know that $\Lambda(L^2(Q))$ is a bounded set in $H^{3/2,1/4}(Q)$. From Proposition \ref{5.x.} we know that  $\Lambda(L^2(Q))\hookrightarrow L^2(Q)$ is a compact embedding , then by Schauder fixed point theorem, it exists $\hat{y}\in L^2(Q)$ such that 
  
   \begin{equation}\label{eq1lin2}
    \begin{array}[c]{lll}
        \hat{y}_t-\Delta    \hat{y} +F(   \hat{y}) =0 & \text{in} \, Q, \\
       \partial_\nu    \hat{y}+b   \hat{y}= f_z1_\gamma & \text{on} \,\Sigma, \\
          \hat{y}(0)=y_0,    \hat{y}(T)=0  & \text{in} \, \Omega.\\
    \end{array}
\end{equation}

Once the existence of the solution is established, it remains to show that $\hat{f}$ solves the optimization problem \eqref{opnull}.
\textbf{Step 1}. Let $f_n\in \mathcal{F}$ a minimizing sequence of the functional $S(\cdot)$ i.e 
\begin{equation}
    \liminf_{n\rightarrow \infty} S(f_n)=\inf _{f\in \mathcal{F}}S(f).
\end{equation}
Since  $f_n\in \mathcal{F}$ is a minimizing sequence, it is bounded in $\mathcal{F}$. Therefore, there exists a subsequence $f_{n_k}$ that is weakly convergent to $\tilde{f}\in \mathcal{F}$ . Then the associated solution  $y_n$ has a subsequence  that converges to $\tilde{y}$ that solves \eqref{eq1lin}.  Then because the functional $S(v,\cdot)$ is lower semicontinous 
\begin{equation}
    S(\tilde{f})\leq \liminf_{n \rightarrow \infty }S(f_{n_k})=\inf_{f\in \mathcal{F}}
    S(f).
\end{equation}
 
\end{proof}
\textbf{Step 3:} Once it is ensured the existence of a minimizer for the functional $S(v,\cdot)$, we proceed to compute explicit solutions for the control problem. Define $\bar{y}\in C(0,T;H^1_0(\Omega))\cap C^1(0,T;L^2(\Omega))$  as  the solution to the system
		\begin{equation}
					\left\{\begin{array}[c]{lll}
						\bar{y}_t-\Delta \bar{y}=0 &\hbox{in} & Q,\\
                        \bar{y}=0 & \hbox{on} & \Sigma,\\
                        \bar{y}(0)=y_0 &\hbox{in} & \Omega.
				\end{array}\right.
		\end{equation}
Define the map $H_0:L^2(Q)\longrightarrow L^2(Q)$ as $H_0(q)=z$ where $z$ is the solution to the problem
			\begin{equation}\label{.fg3}
						\left\{	\begin{array}[c]{lll}
								z_t-\Delta z=q &\hbox{in} & Q,\\
                        \partial_\nu z +bz=0 & \hbox{on} &  \Sigma, \\
                        z(0)=0 &\hbox{in} & \Omega.
						\end{array}\right.
			\end{equation}
Observe that $H^*_0:L^2(Q)\longrightarrow L^2(Q)$ is given by $H^*_0(\psi)=\varphi$ where $\varphi$ solves the equation
				\begin{equation}\label{t--g}
							\left\{\begin{array}[c]{lll}
								-\varphi_t-\Delta \varphi =\psi &\hbox{in} & Q,\\
                                \partial _\nu \varphi+b\varphi =0  &\hbox{on} &\Sigma,\\
                                \varphi(T)=0  &\hbox{in} & \Omega.
						\end{array}\right.
				\end{equation}
            For each $q\in L^2(Q)$, it holds that  $H_0(q)\in L^2(0,T;H^1_0(\Omega))\cap H^1(0,T;L^2(\Omega))$ . Define the boundary operator $G:L^2(\Sigma)\to H^{1/2,1/4}(Q)\subset L^2(Q)$ given by $G(\beta)=f$ , where $f$ is a solution to the boundary value problem
		\begin{equation}
				\left\{	\begin{array}[c]{lll}
						f_t-\Delta f=0 &\hbox{in} & Q,\\
                        \partial_\nu f+bf=\beta & \hbox{on} &  \Sigma, \\
                        f(0)=0 &\hbox{in} & \Omega.	
				\end{array}\right.
		\end{equation}
  By Proposition \eqref{5.x.}, the operator $G:L^2(\Sigma)\longrightarrow L^2(\Omega)$ is compact. Define the map $M:L^2(0,T;H^1(\Omega))\times L^2(\Sigma)\longrightarrow  L^2(Q)$ given by 

	\begin{equation}\label{mqas}
			M(y,f)=y+H_0(F(y))-G(f1_{\gamma})-\bar{y}.
	\end{equation}
   
This map is compact and it is straightforward to verify that the condition $M(y,f)=0$  is equivalent to \eqref{eq1}.  By the above arguments, the null control minimized the $L^2$ functional; thus, it remains to derive the explicit form of the control. To this end, consider the following optimization problem 		\begin{equation}\label{p54t}
        \begin{array}[c]{lll}
          \displaystyle  \text{min}\frac{1}{2}\int_{Q}\varrho^2|y|^2 +\frac{1}{2}\int_{\gamma\times (0,T)}\varrho_0^2|f|^2\\
 		  M(y,f)=0\, ,(y,f)\in \mathcal{Y}\times \mathcal{F}.
        \end{array}
 		\end{equation}    
The operators $H_0$ and $G$ are of class $C^1$ so it is the operator $M$. Restrict the operator $M$ to the subspace$\mathcal{Y}\times \mathcal{F}$. Given any directions $(z,g) \in\mathcal{Y}\times \mathcal{F}$ the derivative of $M$ is the operator $M'$  given by
  		\begin{equation}
  				M'(y,f)(z,g)=z+H_0(F'(y)z)-G(g1_{\gamma}).
  		\end{equation}
  Since $M$ is compact, also $M'(y,f)$ is compact and it is possible to conclude that $\text{ker} M'(y,f)^\bot=\text{Rank} M'(y,f)^*$. It is necessary to compute the dual operator .
  \begin{lemma}
  Let $M'(y,f)^*:L^2(Q)\longrightarrow \mathcal{Y}\times \mathcal{F}$. Its dual is given by

			\begin{equation}\label{s.gh}
    M'(y,f)^*(\psi)=(H_0^*(\psi),-G^*(\psi)1_{\gamma}).
\end{equation}
	
\end{lemma}

 Let $(y,f)\in \mathcal{Y}\times \mathcal{F}$  a solution to \eqref{p54t}. Consider $\lambda >0$ and $(h,k)\in \text{Rank}(M'(y,f)^*)$ given by the Dubovitsky-Milyoutin formalism in Proposition \ref{dv.yu}.  For $\psi\in L^2(Q)$ such that $(h,k)=M'(y,f)^*(\psi)$, observe that after normalizing $\lambda=1$, \eqref{1.3gaw} takes the form
 	\begin{equation}\label{ehgjy684}
 					(\varrho^2y,\varrho^2_0f)+\lambda (H_0^*(\psi),G^*(\psi)1_{\gamma})=0.
 	\end{equation}
  Define  $p:=H_0^*(\psi)$ that satisfies $L_0^*(p)=\psi$.  It is possible to deduce that 
    		\begin{equation}\label{1.d}
    				\begin{array}[c]{lll}
    					\displaystyle	y&=\varrho^{-2}H_0^*(\psi),\\
                            \displaystyle f
                            &=-\varrho_0^{-2}G^*(L_0^*(p))1_\gamma.
    				\end{array}
    		\end{equation}
We find that the pair $(y,f)\in \mathcal{Y}\times \mathcal{F}$  solves the problem \ref{p54t}  and then $p\in \mathcal{P} $.  

\begin{lemma}\label{lemmanm}
    Let $p\in\mathcal{P}$. Then $G^*(L_0^*(p))= p$.
\end{lemma}
 \begin{lemmaproof} 
     Follow a standard density argument. Take $p\in \mathcal{P}_0$ for simplicity. For any $q\in C^\infty(Q)$, integrate $G^*(L_p(p))q$ in $\Sigma$, apply integration by parts and consider that $L_0(G(q))=0$ to get
     \begin{equation}
         \int_\Sigma G^*(L_0^*(p))q\,d\Sigma=\int_\Sigma p\partial_\nu G(q)-\partial_\nu pG(q)\,d\Sigma 
     \end{equation}
Considering boundary conditions 
\begin{equation}
    \int_\Sigma G^*(L_0^*(p))q\,d\Sigma =\int_{\Sigma}pq\,d\Sigma, \,\,\,\,\,\forall q\in C^\infty(Q).
\end{equation}
 \end{lemmaproof}
 
 From this lemma, it is possible to obtain the explicit solution \eqref{sol123}.  Again, since $(y,f[v])\in \mathcal{Y}\times \mathcal{F}$   fulfills  the restrictions of  Problem  \ref{p54t}. By identity \eqref{1.d} it is possible to obtain the fourth order equation 

\begin{equation}\label{p7.64}
                				\left\{\begin{array}[c]{lll}
                				\displaystyle L_{F_0(y)}\Big(\varrho^{-2}(L^*_{F'(y)}(p))\Big)=0 &\hbox{in} & Q,\\
                                \varrho^{-2}L^*_{F'(y)}(p)=\varrho_0^{-2} p1_{\gamma}&\hbox{on} & \Sigma,\\
                                \varrho^{-2}L^*_{F'(y)}(p)(0)=y_0 &\hbox{in} & \Omega.
                		\end{array}\right.
                \end{equation}
\textbf{Step 4}. Take $p'\in \mathcal{P}$ and multiply by it in equation \ref{p7.64}. Using integration by parts and boundary conditions, the integral form for \eqref{p7.64}  is given by
	\begin{equation}\label{inhgy65}
    \begin{array}[c]{lll}
        	 \displaystyle \int_Q\varrho^{-2}L^*_{F'(y}(p)L_{F_0(y)}^*(p')&\displaystyle+\int_{\gamma\times(0,T)}\!\!\!\!\!\!\!\!\!\!\!\!\varrho_0^{-2} p\, p'\,d\Sigma=
            \displaystyle \int_Q gp'\,d\Sigma+\int_\Omega y_0p'(0)dx.   
               \end{array}
		\end{equation}
Then, by choosing $p':=p$, it is possible to get
		\begin{equation}\label{biy576}
        \begin{array}[c]{lll}
				\displaystyle\int_Q\varrho^{-2}L^*_{F'(y)}(p)L_{F_0(y)}^*(p)+\int_{\gamma \times(0,T)}\varrho_0^{-2}| p|^2\, d\Sigma =
          \displaystyle      \int_{Q} gp\,dxdt+\int_\Omega y_0p(0)dx .
                \end{array}
		\end{equation}
Now, it is necessary to make some estimates from \eqref{biy576}. We first analyze  estimates for the solution of the Helmholtz equation.  By inequality

For this, apply  H\"older and Young inequalities. First, it is important to note that the weight $\varrho$ is bounded in the interval $[0,T/2]$ and it is possible to see that $\|p(0)\|_{L^2(\Omega)}\leq \max_{t\in [0,T/2]}\|p(t)\|_{L^2(\Omega)}$. Denote $M_1=\sup_{y\in \mathbb{R}}|F'(y)|$. From Theorem \ref{carleman1} fixing $\lambda>\lambda_0 $ and $s>s_0$ the following inequality holds : 
\begin{equation}
s^{3/2}>M\sqrt{2}\lambda^{-2}\sup_{Q}\xi^{-3/2}.
\end{equation}

Then, define $S:=\sup_Q\frac{\varrho_0}{\varrho}< 1/(M\sqrt{2})$ . It is not difficult to see that

\begin{equation}
\int_{Q}\varrho^{-2}|p|^2dxdt\leq S^2\int_Q\varrho_{0}^{-2}|p|^2\,dxdt.
\end{equation}
\begin{lemma}
    
\end{lemma}

There exists a positive number $\beta$ such that 
\begin{equation}
    \frac{1}{M}\frac{S^2M^2}{1-M^2S^2}<\beta<\frac{1}{M}.
\end{equation}

.Then 
     
		\begin{equation}\label{ine.75648}
				\begin{array}[c]{lll}
				B(0;p,p)&= \displaystyle \int_\Omega y_0(x)p(x)\,dx+\int_{Q}gp\,dxdt \\
                
                & \displaystyle- \int_Q \varrho^{-2}\left(F_0(y)L^*_0(p)p+F'(y)L_0^*(p)p+F'(y)F_0(y)|p|^2\right)\,dxdt \\
                & \displaystyle - \int_Q \varrho^{-2}\left(F_0(y)L^*_0(p)p+F'(y)L_0^*(p)p+F'(y)F_0(y)|p|^2\right)\,dxdt \\
                &  \displaystyle \leq \|y_0\|_{L^2(\Omega)}\max_{t\in[0,T/2]}\|p(t)\|_{L^2(\Omega)} +\| g\|_{L^2(Q)}\|p\|_{\mathcal{P}}+M^2\int_{Q}\varrho^{-2}|p|^2 \,dxdt\\
                & \displaystyle + 2M\int_Q\varrho^{-2}|L_0^*(p)||p|\,dxdt.  \\               
          \end{array}
          \end{equation}
           Then, by Young inequality  with parameter $\beta$ it is possible to bound
           					\begin{equation}
           							 2M\int_Q\varrho^{-2}|L_0^*(p)||p|\,dxdt \leq M\beta\int_Q\varrho^{-2}|L^*_0(p)|^2\,dxdt+\frac{M}{\beta}\int_{Q}\varrho^{-2}|p|^2\,dxdt,
           					\end{equation}
Since $\varrho_0< S\varrho$, we have

			\begin{equation}
						\int_Q\varrho^{-2}|p|^2\,dxdt\leq S^2\int_Q\varrho_0^{-2}|p|^2\,dxdt.
			\end{equation}

In conclusion, 
          \begin{equation}
          		\begin{array}[c]{lll}
                |B(0,p,p)|\le &\displaystyle C\left(\|g\|_{L^2(Q)}+\|y_0\|_{L^2(\Omega)}\right)B(0,p,p)^{1/2}+M\beta\int_{Q}\varrho^{-2}|L_0^*(p)|^2\,dxdt+\left(M^2+\frac{M}{\beta}\right)S^2\int_{Q}\varrho_0^{-2}|p|^2\\
                &\le \displaystyle C\left(\|g\|_{L^2(Q)}+\|y_0\|_{L^2(\Omega)}\right)B(0,p,p)^{1/2}+
                \max\left\{ \beta M,\left(M^2+\frac{M}{\beta}\right)S^2\right\}B(0,p,p).
				\end{array}
		\end{equation}
        
 Remember that $\beta M<1$, and   from inequality $\frac{1}{M}\frac{S^2M^2}{1-M^2S^2}<\beta$  we get $(M^2+M/\beta)S^2<1$ and the term  $$\max\left\{ \beta M,\left(M^2+\frac{M}{\beta}\right)S^2\right\}B(0,p,p) $$ can be absorbed to the left hand side to get 
     
    		\begin{equation}
    				\begin{array}[c]{lll}
   \|f[v]\|_{\mathcal{F}}+\|y\|_{\mathcal{Y}} <C\left(\|y_0\|_{L^2(Q)}+\|g\|_{L^2(Q)}\right).
    				\end{array}
    		\end{equation}

   \subsection{Fixed point iterations.}
In view of  Proposition \ref{nullsemi}, we can find solutions to the control problem via fixed-point iterative methods 
\begin{enumerate}
    \item Choose $y_1\in \mathcal{Y}$. 
    \item Compute $p_1\in \mathcal{P}$ via \eqref{eqp} considering the  potentials are $F'(y_1)$ and $F_0(y_1)$.
    \item Compute $y_{n+1}\in \mathcal{Y}$ via this recurrence.
Therefore, it  is possible to estimate
 \begin{equation}\label{9--.n}
 \|p_n\|_{\mathcal{P}}+\|y_n\|_{\mathcal{Y}}\leq C\left(\|g\|_{L^2(Q)}+\|y_0\|_{L^2(\Omega)}\right). 
 \end{equation}
Then, by using  compactness argument \cite{simon1986compact}, it is possible to get a subsequence $y_{n_k}\rightarrow y$ strongly in $L^2(Q)$. 

\end{enumerate}
 \section{Conclusions and further problems.}

 An  interesting and difficult  problem is the  null controllability  where a non linearity affects the boundary. Consider
\begin{equation}\label{eq111}
    \begin{array}[c]{lll}
        y_t-\Delta y +ay =g & \text{in} \, Q, \\
       \partial_\nu y+F(y)= f1_\gamma & \text{on} \,\Sigma, \\
       y(0)=y_0  & \text{in} \, \Omega.\\
    \end{array}
\end{equation}
 This case entails additional difficulties, since a fixed-point theorem must be applied on the boundary. To the best of the authors’ knowledge, this problem has not been solved for nonlinearities $F$ with Lipschitz or logarithmic growth.  The problem with two controls where a hierarchy is imposed (\cite{calsavara2022new},\cite{araruna2018hierarchic},\cite{araruna2020hierarchical}) is very interesting and presents more challenges than one control problem
        
\appendix  

\section{Appendix: proof of Lemma \ref{lem1}}

\begin{proof}
    The proof follows the same steps as done in \cite{boundnullchorfi}. Define the function $\psi =e^{-s\alpha}p$ and calculate $\psi_t-\Delta\psi$ to get 

    \begin{equation}\label{carl34}
    M_1\psi+M_2\psi=f, \hspace{1cm} N_1\psi+N_2\psi=0,
\end{equation}
where 
    \begin{equation}
    \begin{array}[c]{lll}
        M_1\psi= -s\lambda^2\xi^\psi|\nabla \eta|-2s\lambda\xi\nabla \psi\cdot\nabla \eta+\partial_t\psi,\\
        M_2\psi = s^2\lambda^2\xi^2\psi|\nabla \psi|^2+\Delta \psi+s\psi\partial_t\alpha,
        \end{array}
    \end{equation}
    and
    \begin{equation}
        N_1\psi=-2s\lambda\xi\psi\partial_\nu \eta,\hspace{1cm}  N_2=\partial_\nu\psi, \hspace{0.5cm} \text{on}\Sigma,
    \end{equation}
    and finally
    \begin{equation}
     f=s\lambda \xi \Delta \eta\psi-s\lambda^2\xi|\nabla \eta|^2\psi
\end{equation}
Then, is possible to get the identities
\begin{equation}\label{cfv}
\begin{array}[c]{lll}
\|f\|^2_{L^2(Q)}+\|g\|^2_{L^2(Q)}&=\|M_1\psi\|^2_{L^2(Q)}+\|M_2\psi\|^2_{L^2(Q)}\|N_1\psi\|^2_{L^2(Q)}+\|N_2\psi\|^2_{L^2(Q)}\\
&+\displaystyle\langle M_1\psi,M_2\psi\rangle_{L^2(Q)} +\displaystyle\langle N_1\psi,N_2\psi\rangle_{L^2(\Sigma)}.
\end{array}
\end{equation}
It is important to note that 
\begin{equation}
  \langle N_1\psi,N_2\psi\rangle_{L^2(\Sigma)}=-2s\lambda\int_{\Sigma}\xi\partial_\nu\eta\psi\partial_\nu\psi\,d\Sigma.
\end{equation}

 The computation to show $\langle M_1\psi,M_2\psi\rangle_{L^2(Q)}$ explicitly follows exactly the same steps as done in  \cite{boundnullchorfi} \footnote{For simplicity we will use exactly the same notation  to guide the reader through the equations}. Recall that  $\partial_\nu \psi=s\lambda\xi \psi \partial_\nu \eta$ on $\Sigma$ it is possible to get
\begin{equation}
\begin{array}[c]{lll}
&\displaystyle\langle M_1\psi,M_2\psi \rangle_{L^2(Q)}=\\ =&\displaystyle\sum_{k=1}^{14}I_k+\sum_{k=1}^4B_{1,k}= s^{3}\lambda^{4}
\int_{\Omega_T} |\nabla \eta|^{4}\xi^{3}\psi^{2}\,dx\,dt
+ s^{3}\lambda^{3}
\int_Q \Delta\eta\,|\nabla \eta|^{2}\xi^{3}\psi^{2}\,dx\,dt \\
& \displaystyle+s^{3}\lambda^{3}
\int_Q \nabla(|\nabla \eta|^{2})\cdot\nabla\eta\,\xi^{3}\psi^{2}\,dx\,dt
-s^{2}\lambda^{2}
\int_{Q} |\nabla \eta|^{2}\partial_t\xi\,\xi\psi^{2}\,dx\,dt 
\displaystyle -s^{2}\lambda^{2}
\int_{Q} |\nabla \eta|^{2}\partial_t\alpha\,\xi\psi^{2}\,dx\,dt\\
&\displaystyle+ s^{2}\lambda
\int_{Q} \nabla(\partial_t\alpha)\cdot\nabla\xi\,\psi^{2}\,dx\,dt 
 \displaystyle+ s^{2}\lambda
\int_Q \partial_t\alpha\,\xi\,\Delta\eta\,\psi^{2}\,dx\,dt
- \frac{s}{2}
\int_Q \partial_t^{2}\alpha\,\psi^{2}\,dx\,dt \\
& \displaystyle+ d^{2}s\lambda^{2}
\int_Q \xi|\nabla \eta|^{2}|\nabla\psi|^{2}\,dx\,dt
+ 2s\lambda^{2}
\int_Q \xi\psi\,\nabla(|\nabla \eta|^2)\cdot\nabla\psi\,dx\,dt
 + 2s\lambda^3
\int_Q |\nabla \eta|^{2}\psi\,\xi\,\nabla\eta\cdot\nabla\psi\,dx\,dt\\
&\displaystyle+ 2s\lambda
\int_{\Omega_T} \xi(\nabla^{2}\eta\,\nabla\psi)\cdot\nabla\psi\,dx\,dt 
 \displaystyle- d^{2}s\lambda
\int_Q \xi\,\Delta\eta\,|\nabla\psi|^{2}\,dx\,dt
+ 2s\lambda^ 2
\int_Q \xi|\nabla\eta\cdot\nabla\psi|^{2}\,dx\,dt \\
& \displaystyle-s^{3}\lambda^{3}
\int_{\Sigma} |\nabla\eta|^{2}\xi^{3}\partial_\nu\eta\,\psi^{2}\,d\Sigma
- 2s^2\lambda^3
\int_\Sigma |\nabla\eta|^2\xi^2|\psi|^2\,d\Sigma
 - s^3\lambda^3
\int_\Sigma \xi^3|\partial_\nu \eta|^2|\psi|^2\,d\Sigma
+ s\lambda
\int_\Sigma \partial_\nu\eta\,|\nabla\psi|^{2}\xi\,d\Sigma \\
& \displaystyle
-s^2\lambda^2\int_{\Sigma} \xi^2\psi\nabla\psi\cdot\nabla \eta\,d\Sigma
- s\lambda\int_{\Sigma} \xi_t\partial_\nu\eta |\psi|^2\,d\Sigma 
- s^{2}\lambda
\int_{\Sigma} \partial_t\alpha\,\xi\,\partial_\nu\eta\,\psi^{2}\,dS\,dt .
\end{array}
\end{equation}
Now we need the next lemma
\begin{lemma}
    The next equalities holds
    \begin{equation}
        \xi\leq CT^2\xi^2,\,\,\,\,\,|\nabla \alpha|<C\lambda\xi,\,\,\,\,|\alpha_t|+|\xi_t|\leq CT\xi^2, \,\,\, |\alpha_{tt}|^2\leq CT^2\xi^2.
    \end{equation}
\end{lemma}

Then, we have for $s,\lambda$ long enough 
\begin{equation}
    \begin{array}[c]{lll}
     \displaystyle \sum_{k=1}^{14}I_k &\geq\displaystyle C\left( s^3\lambda^4\int_Q\xi^3|\psi|^2\,dxdt-s^2\lambda^2\int_Q\xi^3|\psi|^2\,dxdt+s\lambda^2\int_Q\xi|\nabla\psi|^2\,dxdt\right.\\
      &\left.\displaystyle-s^2\lambda^4\int_Q\xi^2|\psi|^2\,dxdt -s\lambda\int_Q\xi|\nabla \psi|^2\,dxdt  \right),\\
      &\displaystyle\geq C\left( s^3\lambda^4\int_Q\xi^3|\psi|^2\,dxdt+s\lambda^2\int_Q\xi|\nabla\psi|^2\,dxdt\right)
    \end{array}
\end{equation}

Now we have
\begin{equation}
    -s^{3}\lambda^{3}
\int_{\Sigma} |\nabla\eta|^{2}\xi^{3}\partial_\nu\eta\,\psi^{2}\,d\Sigma\geq -s^{3}\lambda^{3}
\int_{\Gamma/\gamma\times(0,T)}\!\!\!\!\!\!\!\xi^{3}\partial_\nu\eta\,\psi^{2}\,d\Sigma-s^{3}\lambda^{3}
\int_{\gamma\times (0,T)}\!\!\!\!\!\!\!\!\! \xi^{3}\,\psi^{2}\,d\Sigma.
\end{equation}
Next 
\begin{equation}
     - s^3\lambda^3
\int_\Sigma \xi^3|\partial_\nu \eta|^2|\psi|^2\,d\Sigma \geq -  s^3\lambda^3c^2_0
\int_\Sigma \xi^3|\psi|^2\,d\Sigma - s^3\lambda^3\|\partial_\nu\eta\|^2_\infty
\int_{\gamma\times(0,T)}
 \xi^3|\psi|^2\,d\Sigma
\end{equation}
\begin{equation}
     s\lambda
\int_\Sigma \partial_\nu\eta\,|\nabla\psi|^{2}\xi\,d\Sigma \geq   s\lambda
\int_\Sigma \partial_\nu\eta\,|\nabla\psi|^{2}\xi\,d\Sigma
\end{equation}
Proceed in the same way to bound the other boundary integrals. From these estimates it is possible to get 
\begin{equation}
\begin{array}[c]{lll}
    \sum_{k=1}^7B_k &\geq \displaystyle  Cs^2\lambda^3\int_\Sigma\xi^2|\psi|^2\,dxdt +Cs^3\lambda^3\int_\Sigma\xi ^3|\nabla \psi|^2\,dxdt
         +Cs^{-1}\int_{\Sigma}\xi^{-1}|\Delta  \psi|^2\,d\Sigma\\
         &-\displaystyle s^3\lambda^3\int_{\gamma\times (0,T)} \xi^3|\psi|^2\,d\Sigma-Cs\lambda \int_{\gamma\times (0,T)}\xi|\nabla \psi|^2\,d\Sigma 
         \end{array}
\end{equation}

Then we have
\begin{equation}
    \begin{array}[c]{lll}
        \displaystyle &\langle M_1\psi,M_2\psi \rangle_{L^2(Q)}\displaystyle+ \langle N_1\psi,N_2\psi \rangle_{L^2(\Sigma)}+ \\
       &\displaystyle+\displaystyle s^3\lambda^3\int_{\gamma\times (0,T)} \xi^3|\psi|^2\,d\Sigma+Cs\lambda \int_{\gamma\times (0,T)}\xi|\nabla \psi|^2\,d\Sigma \geq\\
       & \displaystyle Cs^2\lambda^3\int_\Sigma\xi^2|\psi|^2\,dxdt +Cs^3\lambda^3\int_\Sigma\xi ^3|\nabla \psi|^2\,dxdt
         + Cs^{-1}\int_{\Sigma}\xi^{-1}|\Delta  \psi|^2\,d\Sigma\\
         &+\displaystyle s^3\lambda^4\int_Q\xi^3|\psi|^2\,dxdt+s\lambda^2\int_Q\xi|\nabla\psi|^2\,dxdt
    \end{array}
\end{equation}

Then, we have that 
\begin{equation}
    \displaystyle Cs^{-1}\int_{\Sigma}\xi^{-1}|\Delta  \psi|^2\,d\Sigma\leq \frac{1}{2}\|N_1\psi\|^2+Cs^3\lambda^2\int_{\Sigma}\xi^3|\psi|^2\,d\Sigma+Cs^3\lambda^4\int_{\gamma\times(0,T)}\!\!\!\!\!\!\!\!\!\!\xi^3|\psi|^2\,d\Sigma.
\end{equation}

From \eqref{cfv} and  the above equation we have
\begin{equation}\label{sdf}
    \begin{array}[c]{lll}
  \|M_1\psi\|^2_{L^2(Q)}&\displaystyle+\|M_2\psi\|^2_{L^2(Q)}  +\|N_1\psi\|^2_{L^2(\Sigma)}+\|N_2\psi\|^2_{L^2(\Sigma)}+Cs^2\lambda^3 \int_{\Sigma}\xi^2|\psi|^2\,d\Sigma \\
  &\displaystyle+Cs^3\lambda^3\int_{\Sigma}\xi^3|\psi|^2 \,d\Sigma +s^3\lambda^4\int_Q\xi^3|\psi|^2\,dxdt+s\lambda\int_Q\xi|\nabla \psi|\,dxdt \\
  &\displaystyle   + \displaystyle s^3\lambda^3\int_{\gamma\times (0,T)}.\xi^3|\psi|^2\,d\Sigma 
    \end{array}
\end{equation}
Following the same arguments as in \cite{guerrero} we can estimate the first two terms of \eqref{sdf} to get
\begin{equation}\label{sdf2}
    \begin{array}[c]{lll}
\displaystyle s^{-1}\int_Q \xi^{-1}(|\psi_t|^2+|\Delta \psi|^2)\,dxdt&\displaystyle+Cs^2\lambda^3 \int_{\Sigma}\xi^2|\psi|^2\,d\Sigma \\
  &\displaystyle+Cs^3\lambda^3\int_{\Sigma}\xi^3|\psi|^2 \,d\Sigma +s^3\lambda^4\int_Q\xi^3\lambda^4\,dxdt+s\lambda\int_Q\xi|\nabla \psi|\,dxdt \\
         &\displaystyle\leq \displaystyle+s\lambda T\int_{\Sigma}\xi^2|\psi|^2\,d\Sigma
        \displaystyle + \displaystyle s^3\lambda^4\int_{\gamma\times (0,T)} \xi^3|\psi|^2\,d\Sigma \\
      & \displaystyle \leq s^3\lambda^4\int_{\gamma\times(0,T)}\xi^3|\psi|^2\,dxdt
    \end{array}
\end{equation}

Only remain to estimate the first integral respect to $\varphi$ to have the desired inequality.  We can see that $|\psi_t|^2\geq |e^{-s\alpha}p_t|^2$ and the same for the gradient  $\nabla \psi$, so from this conclusion it is possible to get from \eqref{sdf2} the desired inequality. Finally using $|F_1|=|F_1-\nabla \cdot F_2+\nabla \cdot  F_2|$ and by Young inequality is possible to get the term
\begin{equation}
    \int_Q\varrho^{-2}|F_1|^2\,dxdt\leq \int_{Q}\varrho^{-2}|p_t+\Delta p|^2\,dxdt +\int_Q \varrho^{-2}|F_2|^2\,dxdt
\end{equation}
That consider the desired form of the inequality.

\end{proof}

\bibliographystyle{unsrt}
\bibliography{biblio}
\end{document}